\documentclass[reqno,12pt]{amsart}%
\usepackage{amsfonts}
\usepackage{amsmath}
\usepackage{amssymb}

\def\bA{{\mathbf{A}}}

\def\cC{{\mathcal{C}}}

\def\ep{\varepsilon}

\def\dW{{\dot{W}}}

\def\FF{\Phi}

\newtheorem{theorem}{Theorem}
\theoremstyle{plain}

\newtheorem{definition}[theorem]{Definition}
\newtheorem{proposition}[theorem]{Proposition}

\newtheorem{remark}[theorem]{Remark}

\newcommand\bel[1]{\begin{equation}\label{#1}}
\newcommand\ee{\end{equation}}

\numberwithin{theorem}{section}
\numberwithin{equation}{section}

\begin{document}
\title[Stationary PAM]
{Heat Equation With a Geometric Rough Path Potential in One Space Dimension:
Existence and Regularity of Solution}
\author{H.-J. Kim}
\curraddr[H.-J. Kim]{Department of Mathematics, USC\\
Los Angeles, CA 90089}
\email[H.-J. Kim]{kim701@usc.edu}
\urladdr{}
\author{S. V. Lototsky}
\curraddr[S. V. Lototsky]{Department of Mathematics, USC\\
Los Angeles, CA 90089}
\email[S. V. Lototsky]{lototsky@math.usc.edu}
\urladdr{http://www-bcf.usc.edu/$\sim$lototsky}



\begin{abstract}
A solution of the heat equation with a distribution-valued potential is
constructed by regularization.
When the potential is the generalized derivative of a H\"{o}lder
continuous function,  regularity of the resulting solution
is in line with the standard parabolic theory.
\end{abstract}

\maketitle

\today

\section{Introduction}
\label{sec:Intro}

Let $W=W(x)$ be a continuous function on $[0,\pi]$, and denote by
$\dot{W}$ the generalized derivative of $W$:
$$
\int_0^{\pi} W(x)h'(x)dx=-\int_0^{\pi} \dW(x)h(x)dx
$$
for every continuously differentiable $h$ with compact support in $(0,\pi)$.
By direct computations, $\dot{W}$ is a generalized function from
  the closure of  $L_2((0,\pi))$ with respect to the norm
\bel{norm-1}
\|h\|_{-1}^2=\sum_{k\geq 1}\frac{h_k^2}{k^2},\
h_k=\sqrt{\frac{2}{\pi}}\int_0^{\pi} h(x)\sin(kx)dx.
\ee

The objective of the paper is investigation of the equation
\bel{eq:main}
\begin{split}
\frac{\partial u(t,x)}{\partial t} &= \frac{\partial^2 u(t,x)}{\partial x^2}+
u(t,x)\dW(x),   \ 0<t<T,\ 0<x<\pi;\\
u(0,x)&=\varphi(x),\ u(t,0)=u(t,\pi)=0.
\end{split}
\ee

If $\dW$ were a H\"{o}lder continuous function of order $\alpha\in (0,1)$,
then standard parabolic regularity would imply that the classical solution
of \eqref{eq:main} is H\"{o}lder $2+\alpha$ in space and
H\"{o}lder $1+(\alpha/2)$ in time; cf.  \cite[Theorem 10.4.1]{KrylovHolder}.
If $W$ is H\"{o}lder $\alpha$, then it might be natural to expect for the
solution of \eqref{eq:main} to lose one derivative in space and ``one-half
of the derivative'' in time, that is, to become
H\"{o}lder $1+\alpha$ in space and
H\"{o}lder $(1+\alpha)/2$ in time. The objective of the paper is to
establish this result in a rigorous way.

The starting point must be interpretation of the product $u\dW$,
which we do using the ideas from  the rough path theory.
Here is an outline of the ideas.

If $h=h(x)$ is a continuously differentiable function with $h(0)=0$, then,
by the
fundamental theorem of calculus,
\bel{FTC}
\int_0^x h(s)h'(s)ds=\frac{h^2(x)}{2};
\ee
for a {\em continuous} function $W=W(x)$, the integral
$\int_0^x W(s)\dW(s)ds$ is, in general,  not defined.
An extension of the Riemann-Stieltjes construction, due to
Young,  is possible when $W$ is H\"older continuous of
order bigger than $1/2$; then \eqref{FTC} continues to
hold as long as $W(0)=0$. For less regular functions $W$, the value of
$\int_0^x W(s)\dW(s)ds=\int_0^x\int_0^{s}\dW(x_1)dx_1\dW(s)ds$
and possibly higher order
iterated integrals $\int_0^x\int_0^{x_1}\int_0^{x_2}
\dot{W}(s)ds\dot{W}(x_2)dx_2
\dW(x_1)dx_1$, etc. must be {\em postulated}, which is the
subject of the {\tt rough path theory}. The collection of all such
iterated integrals, known as the {\tt signature} of the rough path $W$,
 is encoded in the solution of the ordinary differential equation
\bel{rps}
 \frac{dY(x)}{dx}=Y(x)\dot{W}(x),\ x>0,
 \ee
One particular case of the rough path construction {\em stipulates}
that the traditional rules of calculus, such as \eqref{FTC},  continue to hold.
That is, given a continuous function  $W=W(x),\ x\geq 0$,
 the solution of the ordinary differential equation is {\em defined} to be
 \bel{grps}
 Y(x)=Y(0)e^{W(x)-W(0)};
 \ee
 in what follows we refer to this as the {\tt  geometric rough path (GRP)}
 solution of \eqref{rps}. The main consequence of \eqref{grps} is that
 if $W^{(\ep)}=
W^{(\ep)}(x)$, $\ep>0,$ is a sequence of continuously differentiable functions
such that $\lim\limits_{\ep\to 0}
\sup\limits_{x\in (0,L)} |W(x)-W^{(\ep)}(x)|=0$,
then
$$
\lim\limits_{\ep\to 0} \sup\limits_{x\in (0,L)} |Y(x)-Y^{(\ep)}(x)|=0,\ \
{\rm\  where\ \ }
Y^{(\ep)}(x)=Y(0)e^{W^{(\ep)}(x)-W^{(\ep)}(0)},
$$
that is,
$$
\frac{dY^{(\ep)}(x)}{dx}=Y^{(\ep)}(x)\frac{dW^{(\ep)}(x)}{dx}.
 $$

It is therefore natural to define a GRP solution of \eqref{eq:main}
as a  suitable limit of the solutions of
\begin{align}
\label{eq:CVS-1}
\frac{\partial u^{(\ep)}(t,x)}{\partial t} &=
\frac{\partial^2 u^{(\ep)}(t,x)}{\partial x^2} +
u^{(\ep)}(t,x)\dW^{(\ep)}(x),
\ 0<t< T, \ x\in (0,\pi),\\
\label{eq:IBC}
u^{(\ep)}(t,0)&=u^{(\ep)}(t,\pi)=0,
\ u^{(\ep)}(0,x)=\varphi(x),
\end{align}
where $\varphi\in L_2((0,\pi))$ and $W^{(\ep)}$ are absolutely continuous approximations of $W$ so that
$W^{(\ep)}(x)=\int_0^x \dW^{(\ep)}(s)ds$ and
$\dW_{\varepsilon}\in L_{\infty}((0,\pi))$.

Without further regularity assumptions on $\varphi$ and $\dW^{(\ep)}$,
equation \eqref{eq:CVS-1} can only be interpreted in variational form,
leading to a generalized solution; under additional
assumptions about $\varphi$ and $\dW^{(\ep)}$, equation \eqref{eq:CVS-1}
can have a classical solution. Accordingly, the paper defines and
investigates  the two possible GRP solutions of \eqref{eq:main},
classical (Section \ref{sec:CS}) and generalized (Section \ref{sec:g-rps}).
The last section discusses  extensions of the results to more general
equations.

In physics and related applications,
 equations of the type \eqref{eq:main} can represent
either the Sr\"{o}dinger equation (after making time imaginary)
or the (continuum) parabolic Anderson model.
In both cases, the potential $\dW$ is
often random, to model complexity and/or incomplete knowledge
 of the corresponding Hamiltonian. When the time
 variable is present in $W$, the product $u\dW$ can often be interpreted as the
 It\^{o} stochastic integral; with only space parameter, such an interpretation
 is not possible and  the Wick-It\^{o}-Skorokhod integral becomes one
 of the available options.

 The method of regularization, that is, using an approximation of $W$ with
 smooth functions, is a different approach. In the stochastic setting, it
 corresponds to the Stratonovich (or Fisk-Stratonovich) integral;
 cf. \cite{WZ65}. In a more recent theory of rough paths, it corresponds to
 the geometric rough path formulation of the equation.

 Paper \cite{Hu15} presents a comprehensive study of \eqref{eq:main}
 in the whole space for a large class of random potentials $\dW$, both with and without time parameter and under various interpretations of $u\dW$,
 including It\^{o}, Wick-It\^{o}-Skorokhod, and Stratonovich.
 This generality leads to less-than-optimal H\"{o}lder regularity
 (namely, $(1+\alpha)/2)$ in time and $\alpha/2$ in space)
 in the particular case of one space dimension. With the Wick-It\^{o}-Skorokhod
 interpretation, the optimal regularity for \eqref{eq:main} and similar equations
 was established  in \cite{KL17}.

 To simplify the presentation,  the current paper considers \eqref{eq:main}
 on a bounded interval with zero boundary conditions; Section \ref{sec:FD}
 outlines possible generalizations. Both classical and generalized solutions
 of \eqref{eq:main}  are constructed, as well as the fundamental solution.
  The limiting procedure via approximations of $\dW$ by regular
  (as opposed to generalized) functions provides
  a clear connection between various constructions. The special structure of the
 equation leads to additional regularity results
  if $W$ is H\"{o}lder continuous:  the classical solution is
 infinitely differentiable in time for $t>0$, and,
  with a particular choice  of the initial
 condition it is also possible to ensure that the solution is
 H\"{o}lder $1+(\alpha/2)$ near $t=0$, that is, better than what is guaranteed by
 the general theory.

 In two or higher space dimensions the
  method of regularization does not work directly because many of the
  limits do not exist. The difficulty is resolved with the help of
 additional constructions using renormalization techniques;
 cf. \cite{Hai15}. As a result, any comparison with one-dimensional setting
 is not especially informative.

Here is the summary of the main notations used in the paper:
 $h'=h'(x)$ denotes the usual
derivative; $\dot{W}$ denotes the generalized derivative;
$u_t$, $u_x$, $u_{xx}$ denote the corresponding partial derivatives
(classical or generalized); to shorten the notations, the interval $(0,\pi)$
is sometimes denoted by $G$.

\section{The Classical Geometric Rough Path Solution}
\label{sec:CS}

We start with an overview of the parabolic H\"{o}lder space.

Given a locally compact
metric space $\mathbb{X}$ with the distance function $\rho$,
denote by $\cC(\mathbb{X})$ the space of real-valued continuous
functions on $\mathbb{X}$. For $\alpha\in (0,1)$, denote
by $\cC^{\alpha}(\mathbb{X})$ the space of H\"{o}lder continuous
real-valued functions on $\mathbb{X}$, that is, functions $F$ satisfying
$$
\sup_{a,b: \rho(a,b)>0}\frac{|F(a)-F(b)|}{\rho^{\alpha}(a,b)}<\infty;
$$
$\cC^{\alpha}(\mathbb{X})$ is a Banach algebra with norm
$$
\|F\|_{\cC^{\alpha}(\mathbb{X})}=\sup_{a}|F(a)|+
\sup_{a,b: \rho(a,b)>0}\frac{|F(a)-F(b)|}{\rho^{\alpha}(a,b)}.
$$
 For a positive integer
$k$, the space $\cC^{k+\alpha}(\mathbb{X})$ consists of real-valued functions
that are $k$ times continuously differentiable, and the derivative of order
$k$ is in $\cC^{\alpha}(\mathbb{X})$.

Of special interest is
$\mathbb{X}=(0,T)\times G$, $G\subseteq\mathbb{R}^{\mathrm{d}},$
with {\em parabolic}  distance
$$
\rho\big((t,x),(s,y)\big)=\sqrt{|t-s|}+|x-y|.
$$
To emphasize the presence of both time and space variables, the
corresponding notations become $\cC^{\alpha/2,\alpha}\big((0,T)\times G\big)$
 and  $\cC^{k+(\alpha/2), 2k+\alpha}\big((0,T)\times G\big)$.
 Also,
 $$
 \|v\|_{\cC^{(1+\alpha)/2, 1+\alpha}\big((0,T)\times G\big)}=
 \sup_{t,x}|v(t,x)|+\|v_x\|_{\cC^{\alpha/2, \alpha}\big((0,T)\times G\big)}
 +\sup_{t\not=s,x}\frac{|v(t,x)-v(s,x)|}{|t-s|^{(1+\alpha)/2}}.
 $$

\begin{definition}
The classical solution of equation
\begin{align}
\label{eq:cl}
U_t(t,x)&=U_{xx}(t,x)+b(t,x)U_x(t,x)+c(t,x)U(t,x)+f(t,x),\
\\
\notag
x\in(0,\pi),\ & t\in (0,T),\ U(0,x)=\varphi(x),\ U(t,0)=U(t,\pi)=0,
\end{align}
is a function $U=U(t,x)$ with the following properties:
\begin{enumerate}
\item for every $t\in (0,T)$ and $x\in (0,\pi)$, the function $U$
is continuously differentiable in $t$, twice continuously differentiable in $x$,
and \eqref{eq:cl} holds;
\item for every $t\in (0,T)$, $U(t,0)=U(t,\pi)=0$;
\item for every $x\in [0,\pi]$, $\lim_{t\to 0+}U(t,x)=\varphi(x)$.
\end{enumerate}
\end{definition}

The following result is well known; cf. \cite[Theorem 10.4.1]{KrylovHolder}.
\begin{proposition}
\label{prop:cls}
If $\varphi\in \cC((0,\pi))$, $\varphi(0)=\varphi(\pi)=0$, and $b,c,f\in \cC^{\beta/2,\beta}\big((0,T)\times (0,\pi)\big)$ for some $\beta\in (0,1)$,
then  equation \eqref{eq:cl} has a  unique classical solution. If in addition
 $\varphi\in \cC^{2+\beta}((0,\pi))$,
 then  $U\in \cC^{1+\beta/2,2+\beta}\big((0,T)\times (0,\pi)\big)$.
 \end{proposition}

 The next result makes it possible to define the classical geometric
 rough path solution of equation \eqref{eq:main}.

 \begin{theorem}
 \label{th:sol1}
 Assume that, for some $\alpha,\beta\in (0,1)$,
  $W^{(\ep)}\in \cC^{1+\beta}((0,\pi))$ for all $\ep>0$,
  $W\in \cC^{\alpha}((0,\pi))$, and $\varphi\in \cC^{1+\alpha}((0,\pi))$.
   If
 $$
 \lim_{\ep\to 0} \|W-W^{(\ep)}\|_{L_{\infty}((0,\pi))}=0,
 $$
 then there exists a function
 $u\in \cC^{(1+\alpha)/2, 1+\alpha}\big((0,T)\times (0,\pi)\big)$ such that
 $$
 \lim_{\ep\to 0} \|u-u^{(\ep)}\|_{L_{\infty}\big((0,T)\times (0,\pi)\big)}=0.
 $$
\end{theorem}

Note that there is no connection between $\alpha$ and $\beta$ in the conditions
of the theorem.

\begin{definition}
\label{def:sol1}
The function $u$ from Theorem \ref{th:sol1} is called the classical GRP
(geometric rough path)  solution of equation \eqref{eq:main}.
\end{definition}

{\bf Proof of Theorem \ref{th:sol1}.}
Let $u^{(\ep)}=u^{(\ep)}(t,x)$ be the classical solution of \eqref{eq:CVS-1}.
Define the functions
$$
H_W^{(\ep)}(x)=\exp\left(\int_0^x W^{(\ep)}(y)dy\right),\
v^{(\ep)}(t,x)=u^{(\ep)}(t,x)H_W^{(\ep)}(x).
$$
By direct computation,
\bel{eq:CVS-00}
\begin{split}
\frac{\partial v^{(\ep)}(t,x)}{\partial t}&=
\frac{\partial^2 v^{(\ep)}(t,x)}{\partial x^2} -
2W^{(\ep)}(x)\frac{\partial v^{(\ep)}(t,x)}{\partial x}+
\big(W^{(\ep)}(x)\big)^2v^{(\ep)}(t,x),\\
& t>0, \ x\in (0,\pi),\\
v^{(\ep)}(t,0)&=v^{(\ep)}(t,\pi)=0, \
v^{(\ep)}(0,x)=\varphi(x)H^{(\ep)}_W(x).
\end{split}
\ee
Define
\bel{HW}
H_W(x)=\exp\left(\int_0^x W(y)dy\right)
\ee
and
let $v=v(t,x)$ be the classical solution of
\bel{eq:CVS-0}
\begin{split}
\frac{\partial v(t,x)}{\partial t}&=\frac{\partial^2 v(t,x)}{\partial x^2} -
2W(x)\frac{\partial v(t,x)}{\partial x}+W^2(x)v(t,x),\ t>0, \ x\in (0,\pi),\\
v(t,0)&=v(t,\pi)=0, \ v(0,x)=\varphi(x)H_W(x).
\end{split}
\ee

Let $V^{(\ep)}(t,x)=v(t,x)-v^{(\ep)}(t,x)$.
We write the equation for $V^{(\ep)}$ as
\bel{eq:difference}
\begin{split}
 V^{(\ep)}_t&=
\widetilde{\bA}_WV^{(\ep)}+V^{(\ep)}W^2+ F^{(\ep)}\\
V^{(\ep)}(0,x)&=\varphi(x)\big(H_W(x)-H^{(\ep)}_W(x)\big),
\end{split}
\ee
where $\widetilde{\bA}_W$ is the operator
$$
h\mapsto h''-2Wh'
$$
with zero Dirichlet boundary conditions on $G=(0,\pi)$,
and
$$
F^{(\ep)}(t,x)=2v_x^{(\ep)}(t,x)\big(W(x)-W^{(\ep)}(x)\big)
+v^{(\ep)}(t,x)\Big(W^2(x)-\big(W^{(\ep)}(x)\big)^2\Big).
$$
Denote by $\widetilde{\FF}_t,\ t>0,$ the semigroup generated by the
operator $\widetilde{\bA}_W$. Then
\eqref{eq:difference} becomes
\bel{eq:difference1}
\begin{split}
V^{(\ep)}(t,x)&=\widetilde{\FF}_t[V^{(\ep)}(0,\cdot)](x)+
\int_0^t \widetilde{\FF}_{t-s}[V^{(\ep)}(s,\cdot)W^2(\cdot)](x)ds\\
&+ \int_0^t \widetilde{\FF}_{t-s}[F^{(\ep)}(s,\cdot)](x)ds.
\end{split}
\ee
The maximum principle for the operator $\widetilde{\bA}_W$ implies
\bel{MPT}
\|\widetilde{\FF}_th\|_{L_{\infty}(G)}\leq \|h\|_{L_{\infty}(G)},\ t>0;
\ee
similarly,
\bel{MPT1}
\sup_{0<t<T}\|(\widetilde{\FF}_th)_x\|_{L_{\infty}(G)}
\leq C_1\big(T,\|W\|_{L_{\infty}(G)}\big)\big(\|h\|_{L_{\infty}(G)}+
 \|h_x\|_{L_{\infty}(G)}\big);
\ee
cf. \cite[Sections 8.2, 8.3]{KrylovHolder}.

By assumption, there exists a positive number $C_0$ such that
$$
\|W\|_{L_{\infty}(G)}\leq C_0,\ \|W^{(\ep)}\|_{L_{\infty}(G)}\leq C_0
$$
for all $\ep>0$.
Define
$$
\Delta_W^{(\ep)}=\|W^{(\ep)}-W\|_{L_{\infty}(G)}.
$$
Then \eqref{eq:CVS-00},  \eqref{MPT}, \eqref{MPT1},
and Gronwall's inequality  imply
$$
\sup_{0<t<T}\|v^{(\ep)}_x\|_{L_{\infty}(G)}(t)
+\sup_{0<t<T}\|v^{(\ep)}\|_{L_{\infty}(G)}(t)
\leq C_2\|\varphi\|_{L_{\infty}(G)},
$$
with  $C_2$ depending only on $C_0$ and $T$.

Also, by the intermediate  value theorem,
$$
\|V^{(\ep)}\|_{L_{\infty}(G)}(0)\leq \|\varphi\|_{L_{\infty}(G)}e^{2\pi C_0}\,\Delta_W^{(\ep)}.
$$
Then we deduce from \eqref{eq:difference1} that
$$
\sup_{0<t<T}\|V^{(\ep)}\|_{L_{\infty}(G)}(t)
\leq C_3\Delta_W^{(\ep)},
$$
with $C_3$ depending only on $C_0,T,$ and $\varphi$.
 It remains to note that
 \bel{eq:MMM}
 u=\frac{v}{H_W},
 \ee
 $$ \ |u-u^{(\ep)}|\leq \frac{|V^{(\ep)}|}{H_W}+
 \frac{|H_W-H^{(\ep)_W}|}{H_W \cdot H^{(\ep)}_W}\, |v^{(\ep)}|
 $$
 and then, with   $C_4$ depending only on $C_0, T$ and $\varphi$,
 \bel{FinalBoundCl}
 \sup_{t\in (0,T),x\in G}|u(t,x)-u^{(\ep)}(t,x)|
 \leq C_4\|W-W^{(\ep)}\|_{L_{\infty}(G)},
\ee
completing the proof of the theorem.
$\hfill\square$

The following is the main result about the classical GRP
solution of \eqref{eq:main}.

\begin{theorem}
\label{th:c-grps}
If $W\in \cC^{\alpha}((0,\pi))$,
$\varphi\in \cC^{1+\alpha}((0,\pi))$, and $\varphi(0)=\varphi(\pi)=0$,
then \eqref{eq:main}
has a unique GRP solution given by
$$
u(t,x)=\frac{\FF_t[H_W\varphi](x)}{H_W(x)},
$$
where $\FF_t$ is the semigroup of the operator
$$
\mathbf{A}_W: h(x)\mapsto h''(x)-2W(x)h'(x)+W^2(x)h(x)
$$
with zero boundary conditions on $(0,\pi)$ and
$H_W$ is from \eqref{HW}. The solution
 belongs to the parabolic H\"{o}lder space
$$
\cC^{(1+\alpha)/2, 1+\alpha}\big((0,T)\times(0,\pi)\big)
$$
and is an infinitely differentiable function of $t$ for $t>0$.
If, in addition,
$\varphi(x)=\psi(x)/H_W(x)$ for some $\psi\in \cC^{2+\alpha}((0,\pi))$,
then $u\in\cC^{1+(\alpha/2), 1+\alpha}\big((0,T)\times(0,\pi)\big). $
\end{theorem}

\begin{proof}
This is a direct consequence of \eqref{eq:MMM} and Proposition \ref{prop:cls}.
The solution is a smooth function of $t$ for $t>0$ because
the coefficients of the operator $\mathbf{A}_W$ do not depend on time
(cf. \cite[Theorem 8.2.1]{KrylovHolder}),
whereas
$\cC^{\alpha}$ regularity of the coefficients of $\mathbf{A}_W$ implies that
the function $\FF_t[\varphi](x)$  cannot be better than $\cC^{2+\alpha}$ in space. Because $H_W\in \cC^{1+\alpha}((0,\pi))$,  a
typical initial condition cannot be better than $\cC^{1+\alpha}((0,\pi))$
and similarly, the solution
cannot be more regular in space than $\cC^{1+\alpha}((0,\pi))$.
If, with a special choice of the initial condition, we ensure that
$\varphi H_W\in \cC^{2+\alpha}((0,\pi))$, then, by Proposition \ref{prop:cls},
 we get better time regularity of the solution near $t=0$.
\end{proof}

 \begin{remark}
 If $W$ is a sample trajectory of the standard Brownian motion, then, with
 probability one,
 $W\in \cC^{1/2-}((0,\pi))$, that it, $W$ is H\"{o}lder continuous of
 every order less than $1/2$. By Theorem \ref{th:c-grps}, with
 $\alpha<1/2$,  we conclude that,
 with probability one, a {\em typical}
 classical GRP solution of the corresponding equation
 \eqref{eq:main} is $\cC^{3/4-}$ in time and $\cC^{3/2-}$ in space;
 with a special choice of the initial condition, it is possible to achieve
 $ \cC^{5/4-}$ regularity in time. Similarly, if $W=B^H$ is fractional
Brownian motion with Hurst parameter $H\in (0,1)$, then a  typical
 classical GRP solution of
 \eqref{eq:main} is $\cC^{(1+H)/2-}$ in time and $\cC^{1+H-}$
 in space.
 \end{remark}

 Inequality \eqref{FinalBoundCl} provides the  rate of convergence of
 the approximate solutions $u^{(\ep)}$ to $u$; this rate
 depends on the particular approximation $W^{(\ep)}$ of $W$.
 As an example, consider
 $$
 W^{(\ep)}(x)=\frac{1}{\ep}\int_0^{\pi}\phi\left(\frac{x-y}{\ep}\right)
 W(y)dy,
 $$
 where
 $$
  \phi\geq 0, \ \phi\in \cC^{\infty}(\mathbb{R}),\
   \phi(x)=0,\ x\notin (-\pi,\pi),\ \int_0^1\phi(x)dx=1,
   $$
  so that  $W^{(\ep)}\in \cC^{\infty}((0,\pi))$.

   By direct computation, if $W\in \cC^{\alpha}((0,\pi))$, then
   $$
   |W(x)-W^{(\ep)}(x)|\leq \ep^{\alpha}\
   \sup_{x\not=y}\frac{|W(x)-W(y)|}{|x-y|^{\alpha}},
$$
   and so
   $$
   \sup_{t,x}|u(t,x)-u^{(\ep)}(t,x)|\leq \tilde{C}_4\ep^{\alpha}.
   $$
   With some extra effort, similar arguments  show that, for every
   $\gamma\in (0,\alpha)$,
   $$
   \|u-u^{(\ep)}\|_{\cC^{(1+\gamma)/2,1+\gamma}\big((0,T)\times(0,\pi)\big)}
   \leq C_5 \ep^{\alpha-\gamma}.
   $$

\section{The Generalized Geometric Rough Path Solution}
\label{sec:g-rps}

The classical solution of \eqref{eq:CVS-1}
might not exist if the initial condition $\varphi$ is not a continuous function, or
if the functions $\dW^{(\ep)}$ are not H\"{o}lder continuous, or if
condition $\varphi(0)=\varphi(\pi)=0$ does not hold. The
generalized solution is an extension of the classical solution using the
idea of integration by parts.

To simplify the notations, we write
$$
G=(0,\pi),\ (g,h)_0=\int_0^{\pi} g(x)h(x)dx,\
\|h\|_0^2=(h,h)_0.
$$

Next, we need an overview of the Sobolev space on $G$.
 The space $H^1_0(G)$ is the closure of the
set of smooth functions with compact support in $G$ with respect to
the norm $\|\cdot\|_1$, where
$$
\|h\|_1^2=\|h\|_0^2+ \|h'\|_0^2.
$$
The space $H^{-1}(G)$, mentioned in the introduction, is a
separable Hilbert space with norm $\|\cdot\|_{-1}$ and is the dual of
$H^1_0(G)$ relative to the inner product $(\cdot,\cdot)_0$; the corresponding
duality will be denoted by $[\cdot, \cdot]_0$.

\begin{definition}
Given bounded measurable functions $a,b,c$ and a generalized function
$f$ from $L_2\big((0,T); H^{-1}(G)\big)$,
 the generalized solution of equation
\bel{gen-sol-basic}
v_t=(v_x+av)_x+bv_x+cv+f,\ t> 0,\ x\in G,
\ee
with initial condition $v|_{t=0}=\varphi\in L_2(G)$ and
zero Dirichlet boundary conditions $v|_{x=0}=v|_{x=\pi}=0$
is an element of $L_2\big((0,T); H^{1}_0(G)\big)$ such that,
for every $h\in H^1_0(G)$ and $t\in (0,T)$,
\begin{equation}
\label{eq:VS-0}
\begin{split}
\big( v,h\big)_0(t)&=\big(\varphi,h\big)_0 -
\int_0^t \big(v_x+av, h'\big)_0(s)ds\\
&+\int_0^t \big(bv_x +c,h\big)_0(s)ds+\int_0^t [f,h]_0(s)ds.
\end{split}
\end{equation}
\end{definition}

The following result is well known; cf. \cite[Theorem III.4.1]{Lady67}.

\begin{proposition}
\label{prop:gs-b}
Let $a=a(t,x),\  b=b(t,x), \ c=c(t,x)$ be bounded measurable functions,
$$
\varphi\in L_2(G),\ f\in L_2\big((0,T); H^{-1}(G)\big).
$$
Then equation \eqref{gen-sol-basic}
has a unique generalized solution and
\begin{align*}
&v\in L_2\big((0,T);H^1_0(G)\big)\bigcap\cC\big((0,T); L_2(G)\big),\\
&\sup_{t\in (0,T)}\|v\|_0^2(t)+\int_0^T
 \|v\|_1^2(s)ds \leq C \Big(
 \|\varphi\|_0^2+\int_0^T \|f\|_{-1}^2(s)ds\Big),
\end{align*}
with $C$ depending only on $T$ and the $L_{\infty}$ norms of $a,b,c$.
\end{proposition}

If  $W^{(\ep)}$ is continuously differentiable and $u^{(\ep)}$
is a classical solution of \eqref{eq:CVS-1},
then we can multiply  both sides by a continuously differentiable
function $h$ with compact support in $G$ and integrate with respect to $t,x$
to get, after integration by parts,
\begin{equation}
\label{eq:VS-1}
\big( u^{(\ep)},h\big)_0(t)=\big(\varphi,h\big)_0 -
\int_0^t \big(u^{(\ep)}_x+u^{(\ep)}W^{(\ep)}, h'\big)_0(s)ds
-\int_0^t \big(u^{(\ep)}_x W^{(\ep)},h\big)_0(s)ds.
\end{equation}

Comparing \eqref{eq:VS-1} and \eqref{eq:VS-0}, we conclude that
$u^{(\ep)}$ is a generalized solution of
\bel{eq:VS-2}
u^{(\ep)}_t=(u^{(\ep)}_x+u^{(\ep)}W^{(\ep)})_x-u^{(\ep)}_xW^{(\ep)}.
\ee
If we now pass to the limit $\varepsilon\to 0$ in \eqref{eq:VS-2}
and assume that all the limits exist and all the equalities
continue to hold, then we get the equation
\bel{grp-gen-0}
u_t=(u_x+uW)_x-u_xW,
\ee
which is a particular case of \eqref{gen-sol-basic}. Before confirming
that this passage to the limit is indeed justified, let us use this non-rigorous
argument to define the generalized GRP solution of \eqref{eq:main}.

\begin{definition}
The generalized geometric  rough path (GRP) solution of \eqref{eq:main}
 is the generalized solution of \eqref{grp-gen-0}.
\end{definition}

The next result is an immediate consequence of Proposition \ref{prop:gs-b}.

\begin{theorem}
\label{th:g-grps}
 If $W\in L_{\infty} ((0,\pi))$ and $\varphi\in L_2((0,\pi))$, then
 \eqref{eq:main} has a unique generalized GRP solution $u$ and
 \begin{align*}
 &u\in L_2\big((0,T);H^1_0((0,\pi))\big)\bigcap\cC\big((0,T); L_2((0,\pi))\big),\\
 &\sup_{t\in (0,T)}\|u\|_0^2+\int_0^T
 \|u\|_1^2(s)ds \leq C \|\varphi\|_0^2,
 \end{align*}
 with  $C$ depending only on $T$ and $\|W\|_{L_{\infty}((0,\pi))}$.
\end{theorem}

Note that  generalized GRP solution  may exist even when $W$ is not continuous.

It remains to confirm that passing to the limit
$\varepsilon\to 0$ in \eqref{eq:VS-2} is justified and
indeed leads to \eqref{grp-gen-0}.

\begin{theorem}
 \label{th:sol2}
 Assume that $W^{(\ep)}\in L_{\infty} (G)$ for all $\ep>0$,
  $W\in L_{\infty} (G)$, and $\varphi\in L_2(G)$.
 Let $u^{(\ep)}$ be the generalized solution of \eqref{eq:VS-1}
 and let $u$ be the generalized solution of \eqref{grp-gen-0}.
 If
 $$
 \lim_{\ep\to 0} \|W-W^{(\ep)}\|_{L_{\infty}(G)}=0,
 $$
 then
 $$
 \lim_{\ep\to 0}
 \Big( \sup_{t\in (0,T)}\|u-u^{(\ep)}\|_0^2(t)+\int_0^T
 \|u-u^{(\ep)}\|_1^2(s)ds\Big)=0.
 $$
\end{theorem}

\begin{proof}
Let $U^{(\ep)}=u-u^{(\ep)}$. Then \eqref{eq:VS-1} and
 \eqref{grp-gen-0} imply that $U^{(\ep)}$ is the generalized solution
 of
 $$
 U^{(\ep)}_t=(U^{(\ep)}_x+U^{(\ep)}W)_x-U^{(\ep)}_xW+
 (u^{(\ep)}-u^{(\ep)}_x)\cdot(W-W^{(\ep)}),\ U|_{t=0}=0.
 $$
 Recall that $C_0$ denotes the common upper bound on the
 $L_{\infty}$ norms of $W$ and $W^{(\ep)}$.
 By Proposition \ref{prop:gs-b},
 \begin{align*}
 &\sup_{t\in (0,T)}\|U^{(\ep)}\|_0^2(t)
 +\int_0^T\|U^{(\ep)}\|_1^2(s)ds\leq
 C_6(T,C_0) \|W-W^{(\ep)}\|_{L_{\infty}(G)}\int_0^T \|u^{(\ep)}\|_1^2(t)dt,
 \\
& \int_0^T \|u^{(\ep)}\|_1^2(t)dt\leq C_7(T,C_0)\|\varphi\|_0^2,
 \end{align*}
 completing the proof.
 \end{proof}

 \begin{remark}
 \label{rm:equiv}
By construction, a classical solution of \eqref{eq:CVS-1} is automatically a
generalized solution. Then Theorems \ref{th:c-grps} and \ref{th:sol2} imply that
a classical GRP solution of \eqref{eq:main},
if  exists, coincides with the generalized GRP solution.
\end{remark}

We will now construct the fundamental GRP solution of equation \eqref{eq:main}.
Let $\mathbf{L}:H^1_0(G)\to H^{-1}(G)$ be the operator
$$
 h\mapsto -\big(h_x+hW\big)_x+h_xW.
$$
Then \eqref{grp-gen-0} becomes
$$
u_t=-\mathbf{L} u.
$$
It is known \cite{Savchuk99} that,
for every $W\in L_{\infty}((0,\pi))$,
\begin{itemize}
\item The operator $\mathbf{L}$
has pure point spectrum;
\item  The eigenvalues $\lambda_k, k\geq 1,$
satisfy
$$
\lim_{k\to \infty}\frac{\lambda_k}{k^2}=1;
$$
\item The corresponding eigenfunctions $\mathfrak{m}_k$ belong  to
$H^1_0((0,\pi))$;
\item The collection $\{\mathfrak{m}_k,\ k\geq 1\}$ can be chosen to
form an orthonormal basis in $L_2((0,\pi))$.
\end{itemize}
 Then, using the functions
 $\mathfrak{m}_k$ in the definition of the generalized solution of
 \eqref{grp-gen-0}, we conclude by Theorem \ref{th:g-grps} that
\bel{FS-sol}
u(t,x)=\sum_{k=1}^{\infty} e^{-\lambda_k t}
(\varphi,\mathfrak{m}_k)_0\, \mathfrak{m}_k(x).
\ee
Equality \eqref{FS-sol} suggests calling the function
$$
\mathfrak{p}(t,x,y)=\sum_{k=1}^{\infty} e^{-\lambda_k t}
\mathfrak{m}_k(x)\mathfrak{m}_k(y)
$$
the {\tt fundamental GRP solution}  of \eqref{eq:main}.
By Theorem \ref{th:c-grps},
$$
\mathfrak{p}(t,x,y)=e^{-\int_0^xW(s)ds}\,\mathfrak{p}_W(t,x,y)\,
e^{\int_0^yW(s)ds},
$$
where $\mathfrak{p}_W$ is the fundamental solution of \eqref{eq:CVS-0}

\section{Further Directions}
\label{sec:FD}

The following extensions are straightforward:
\begin{enumerate}
\item Classical GRP solution for the equation
\bel{fd-gen-1}
u_t=au_{xx}+bu_x+cu+f+u\dot{W}
\ee
 $t\in (0,T),\ x\in (L_1,L_2)$ with H\"{o}lder continuous,
 in $t,x$,  functions $a,b,c,f,$ with
$\inf_{t,x}a(t,x)>0$, and with {\em separated}
 boundary conditions
$$
p_1u(t,L_1)+p_2u_x(t,L_1)=0,\ p_3u(t,L_2)+p_4u_x(t,L_2)=0,
$$
because the results from \cite[Chapter 10]{KrylovHolder} and
\cite[Chapter IV]{Lady67} about solvability of parabolic equations with
H\"{o}lder continuous coefficients still apply.

Note that the change of the unknown function
\bel{ch-var}
u(t,x)=v(t,x)\exp\left(-\int_{L_1}^x W(s)ds\right)
\ee
affects the boundary conditions by changing some of the
 coefficients $p_k$.
\item Generalized GRP solution for  equation \eqref{fd-gen-1}
or for equation
$$
u_t=\big(au_{x}+\tilde{a}u)_x+bu_x+cu+f+u\dot{W},
$$
with zero Dirichlet boundary conditions.
\end{enumerate}

The following extensions are most likely to be possible as well,
but, because the standard parabolic regularity results  (e.g. those
in \cite{KrylovHolder, Lady67}) do not apply, much more effort
could be necessary:
\begin{enumerate}
\item  Classical GRP solution for \eqref{fd-gen-1}
 with general boundary conditions
 \begin{equation}
 \label{fd-GBC}
 \begin{split}
&p_{11}u(t,L_1)+p_{12}u_x(t,L_1)+p_{13}u(t,L_2)+p_{14}u_x(t,L_2)=0,\\
&p_{21}u(t,L_1)+p_{22}u_x(t,L_1)+p_{23}u(t,L_2)+p_{24}u_x(t,L_2)=0.
\end{split}
\end{equation}
Understanding \eqref{fd-GBC}
is necessary, for example, to study \eqref{eq:main}
with periodic boundary conditions
$$
u(t,L_1)=u(t,L_2),\
u_x(t,L_1)=u_x(t,L_2),
$$
after the change of the unknown function according to \eqref{ch-var}.

\item Generalized GRP solution  with boundary conditions other
than zero Dirichlet:  complications start at the integration by parts stage.

\item Equation \eqref{eq:main} on the $(-\infty, +\infty)$ or
$(0,+\infty)$ without assuming that $W$ is bounded:
 while H\"{o}lder regularity is essentially a local property and
should be expected to hold, there are  technical difficulties related to
 the analysis of equation \eqref{eq:CVS-0} unless there is an
additional assumption that $W$ is uniformly bounded.
\end{enumerate}

\def\cprime{$'$}
\providecommand{\bysame}{\leavevmode\hbox to3em{\hrulefill}\thinspace}
\providecommand{\MR}{\relax\ifhmode\unskip\space\fi MR }
\providecommand{\MRhref}[2]{%
  \href{http://www.ams.org/mathscinet-getitem?mr=#1}{#2}
}
\providecommand{\href}[2]{#2}

\end{document}